\documentclass[11pt]{article}
\usepackage[T1]{fontenc}
\usepackage[latin9]{inputenc}
\usepackage{mathtools}
\usepackage{amssymb}
\usepackage{stackrel}
\usepackage{hyperref}
\hypersetup{colorlinks=true, linkcolor=blue, filecolor=magenta, urlcolor=blue}
\usepackage{amsthm,cleveref}
\usepackage[dvipsnames]{xcolor}
\usepackage{mathrsfs}
\usepackage[margin = 3 cm,marginpar= 2cm]{geometry}
\usepackage{amsthm,cleveref}

\makeatletter
\newtheorem{thm}{Theorem}[section]
\newtheorem{defn}[thm]{Definition}

\newtheorem{lem}[thm]{Lemma}
\newtheorem{cor}[thm]{Corollary}

\newtheorem{rem}[thm]{Remark}

\numberwithin{equation}{section}

\@ifundefined{date}{}{\date{}}
\makeatother

\usepackage{babel}

\newcommand{\norme}[1]{\left\Vert #1\right\Vert}
\newcommand{\abs}[1]{\left\lvert #1\right\rvert}

\begin{document}
\title{Derivation of Yudovich solutions of Incompressible Euler from the Vlasov-Poisson system}
\author{
Immanuel Ben-Porat%
     \thanks{Mathematical Institute, University of Oxford, Oxford OX2 6GG, UK.  {Immanuel.BenPorat@maths.ox.ac.uk}} %
\and
    Mikaela Iacobelli%
     \thanks{ETH Z{\"u}rich, Department of Mathematics,
R{\"a}mistrasse 101, 8092 Z{\"u}rich, Switzerland.  {mikaela.iacobelli@math.ethz.ch}} %
     \and
   Alexandre Rege%
     \thanks{ETH Z{\"u}rich, Department of Mathematics,
R{\"a}mistrasse 101, 8092 Z{\"u}rich, Switzerland.  {alexandre.rege@math.ethz.ch}} %
}
\maketitle
\begin{abstract}
We derive the two dimensional incompressible Euler equation
as a quasineutral limit of the Vlasov-Poisson equation using a modulated energy approach. We propose a strategy which enables to treat
solutions where the gradient of the velocity is merely $\mathrm{BMO}$, in accordance to the celebrated Yudovich theorem.
 \end{abstract}

\section{Introduction }

\subsection{General overview}
In this paper we study the derivation of
the 2D incompressible Euler equation as a quasineutral limit from
the Vlasov--Poisson system. The Euler equation on $[0,T]\times\mathbb{T}^{2}$
in vorticity form reads 

\begin{equation}\tag{E}
\partial_{t}\omega+\mathrm{div}(\omega u)=0,\ \omega=\Delta\psi,\ u=(\nabla\psi)^{\bot}.\label{Euler Equation Intro}
\end{equation}
Here $u:[0,T]\times\mathbb{T}^{2}\rightarrow\mathbb{R}^{2}$ is a
vector field referred to as the \textit{velocity field} , $\omega$
is a scalar field on $[0,T]\times\mathbb{T}^{2}$ referred to as the
\textit{vorticity}, and $v^{\bot}$ designates the rotation by $\frac{\pi}{2}$
of a vector $v$, i.e. $v^{\bot}\coloneqq(-v_{2},v_{1})$. 
In the quasineutral scaling, the Vlasov--Poisson system, reads as follows:

\begin{equation}\tag{VP}
\left\{ \begin{array}{lc}
\partial_{t}f^{\varepsilon}+\xi\cdot\nabla_{x}f^{\varepsilon}-\nabla\Phi^{\varepsilon}\cdot\nabla_{\xi}f^{\varepsilon}=0,\vspace{0.25 cm}\\ 
\rho^{\varepsilon}=1-\varepsilon\Delta\Phi^{\varepsilon},\ \rho^{\varepsilon}(t,x)=\int_{\mathbb{R}^{2}}f(t,x,\xi)d\xi.
\end{array}\right.\label{eq:Vlasov Equation}
\end{equation}
Here, the unknown is a time dependent probability density function
$f^{\epsilon}(t,x,\xi)$ on $\mathbb{\mathbb{T}}^{2}\times\mathbb{R}^{2}$.
The variables $(x,\xi)\in\mathbb{\mathbb{T}}^{2}\times\mathbb{R}^{2}$
are referred to as the \textit{space variable} and \textit{velocity
variable} respectively. From a physical point of view, the parameter $\varepsilon$ represents the Debye length, and the quasineutral limits consists in the rigorous study of the above system when $\varepsilon$ tends to zero.

Typically, in the quasineutral limit one
seeks to prove uniform in time weak convergence of the density $\rho^{\varepsilon}(t,x)$
to $1$ and of the current $J^{\varepsilon}(t,x)\coloneqq \int_{\mathbb{R}^{2}}\xi f^{\varepsilon}(t,x,\xi)d\xi$ to the velocity field $u(t,x)$, in the limit as $\varepsilon$ goes to zero. Supplemented with monokinetic initial data, this type of limit offers a derivation of an equation of perfect fluids from a kinetic equation. The mathematical study of the quasineutral limit was initiated in the works of Brenier and Grenier \cite{Brenier Grenier}, \cite{Grenier}, which used an approach based on defect measures. Then in \cite{Grenier2} Grenier established the quasineutral limit of Vlasov--Poisson in the analytic regime. The quasineutral limit was also studied for magnetized plasmas in Golse and Saint-Raymond \cite{Golse Saint Raymond} by employing compactness arguments and in Puel-Saint-Raymond \cite{Puel-S.Raymond}. Brenier \cite{Brenier} used a modulated energy approach to
prove convergence of the density and current associated to the system
\eqref{eq:Vlasov Equation} to $1$ and the velocity field of Euler's equation respectively. 
The notion of solution considered by Brenier is that
of dissipative solutions which was introduced by Lions in \cite{Lions}. 
This result was then extended by Masmoudi in \cite{Masmoudi} to more general initial data. 
More recent works include  \cite{Han-Kwan Iacobelli} which studies a combined mean field quasineutral/gyrokinetic regime and \cite{Ben Porat,Rosenzweig Quantum} which are the semi-classical and quantum counterpart of the latter work. The combined quasineutral/gyrokinetic semi-classical limits have been also the subject of study of the works \cite{Puel}, \cite{Puel2}. We mention that the quasineutral limit can be studied for the Vlasov-Poisson system for ions, which mathematically means that the linear Poisson equation coupled with the Vlasov equation is replaced by a nonlinear version thereof - see  \cite{Han-Kwan},\cite{Griffin-Pickering Iacobelli 2023},\cite{GI20} for more details about this theme. Finally, we refer to \cite{GI22} for an inclusive overview of recent developments in the study of the quasineutral limit.


\subsection{Main contribution}
The focus in this work
is how the amount of regularity that is imposed for the solution
of the limit equation \eqref{Euler Equation Intro} affects the argument leading to the quasineutral limit. More specifically,
in view of the work of Yudovich \cite{Yudovich} about the existence and
uniqueness of weak solutions for the Euler equation \eqref{Euler Equation Intro},
it is natural to raise the question of the derivation of this class of solutions as a quasineutral limit. All the works mentioned so far
did not cover solutions of Yudovich type. To properly understand the
difficulty that emerges when applying the modulated energy
approach introduced by Brenier in \cite{Brenier}, let us briefly recall the strategy of the proof. The argument in \cite{Brenier} is based on a Gr\"onwall
estimate for the following time dependent quantity, known as the \textit{modulated
energy}:

\[
E_{\varepsilon}(t)\coloneqq\frac{1}{2}\int\left|\xi-u(t,x)\right|^{2}f^{\varepsilon}(t,x,\xi)dxd\xi+\frac{\varepsilon}{2}\int\left|\nabla\Phi^{\varepsilon}\right|^{2}(t,x)dx.
\]
After some classical manipulations involving the conservation of energy which can be found in the Appendix of \cite{Brenier}, one obtains the following
expression for the time derivative of $E_{\varepsilon}(t)$: 
\begin{align*}
\dot{E_{\varepsilon}}(t)=-&\int_{\mathbb{T}^{2}}\int_{\mathbb{R}^{2}}d(u)(t,x):(\xi-u(t,x))\otimes (\xi-u(t,x)) f^{\varepsilon}(t,x,\xi)dxd\xi  \\&+\varepsilon\int_{\mathbb{T}^{2}}d(u)(t,x):\nabla\Phi^{\varepsilon} \otimes \nabla\Phi^{\varepsilon} dx \\&+\int_{\mathbb{T}^{2}}A(u)(t,x)\cdot u(t,x)\left(\rho^{\varepsilon}(t,x)-1\right)dx\coloneqq\stackrel[k=1]{3}{\sum}I_{k}.  
\end{align*}
Here, $d(u)$ stands for the symmetric part of the Jacobian matrix
of $u$, i.e. $(d(u))_{ij}\coloneqq\partial_{i}u_{j}+\partial_{j}u_{i}$
and $A(u)\coloneqq\partial_{t}u+u\cdot\nabla u$. Let us examine the
integral $I_{1}$ for instance. When $\nabla u\in L^{\infty}$, it
is noticeable that $I_{1}$ is controlled by the kinetic part of the
energy $E_{\varepsilon}^{1}(t)$. However, when working with Yudovich
solutions, the boundedness of $\nabla u$ is no longer at our disposal.

The aim of this work is to 
show how to treat the more singular integrals that appear in the expression for $\dot{E_{\varepsilon}}(t)$.
To this aim we will introduce a more subtle argument, based on the combination of Fefferman's characterization of
the dual of the Hardy space $\mathscr{H}^{1}$ as the space $\mathrm{BMO}$, with some new estimates on the growth with respect to $\varepsilon$ of the characteristics in Vlasov-type system introduced in \cite{Griffin-Pickering Iacobelli 2023}.
 In particular, a large portion of our proof will be devoted to carefully study how the Hardy norm
of terms of the form 
\begin{equation}
\int_{\mathbb{R}^{2}}(\xi-u(t,x))\otimes (\xi-u(t,x))f^{\varepsilon}(t,x,\xi)d\xi\label{eq:typical integral}
\end{equation}
 grow with respect to $\varepsilon$. In view of the inequality $\left|\int fg\right|\leq\left\Vert f\right\Vert _{\mathrm{BMO}}\left\Vert g\right\Vert _{\mathscr{H}^{1}}$
(or slight variants thereof) this will ultimately allow us to conclude and prove the following:

\begin{thm}[Main result] \label{main thm}
Let the following assumptions on the initial data hold\\ 

$\mathbf{H1}$. $0\leq f_{0}^{\varepsilon}$$\in L^{\infty}\cap L^{1}(\mathbb{T}^{2}\times\mathbb{R}^{2})$
is a probability density.

$\mathbf{H2}.$ $\left\Vert (1+\left|\xi\right|^{k_{0}})f_{0}^{\varepsilon}\right\Vert _{L^{\infty}\cap L^{1}}\leq\overline{M}_{\varepsilon}$
where $k_{0}>4$ and $\overline{M}_{\varepsilon}=O(\varepsilon^{-\alpha})$ for some $\alpha>0$.

$\mathbf{H3}.$ There is some $\beta>0$ such that $E_{\varepsilon}(0)=O({\varepsilon}^{\beta})$.

$\mathbf{H4}.$ $\omega_{0}\in L^{\infty}(\mathbb{T}^{2})$ and $u_{0}$ is the velocity field associated with $\omega_{0}$, i.e. $u_{0}=-\nabla ^{\bot}V\star\omega_{0}$ 
where $V$ is the Coulomb potential on the torus.\\
Let $f^\varepsilon$ be the unique solution to \eqref{eq:Vlasov Equation} with initial data $f_{0}^{\varepsilon}$ provided by \Cref{thm:(Existence-for-Vlasov--Poisson)}, and let $\omega$ be the Yudovich solution to \eqref{Euler Equation Intro} with initial data $\omega_{0}$ provided by \Cref{Yudovich theorem }. Then, there is some
$T_{\ast}=T_{\ast}\left(\left\Vert \omega_{0}\right\Vert _{L^{\infty}(\mathbb{T}^{2})},\alpha, \beta\right)>0$
such that for all $t\in[0,T_{\ast}]$ it holds that 
\[\rho^{\varepsilon}(t,\cdot)\underset{\varepsilon\rightarrow0}{\rightharpoonup}1\]
and 
\[J^{\varepsilon}(t,\cdot)\underset{\varepsilon\rightarrow0}{\rightharpoonup}u(t,\cdot)\] 
weakly in the sense of measures. 
\end{thm}

Let us observe that it is not difficult to construct initial data $f_{0}^{\varepsilon}$ which verify the assumptions $\mathbf{H1}$-$\mathbf{H3}$ - we shall elaborate on this in \Cref{Construction-of-initial data }. The motivation
for this work grew out of a question raised by Rosenzweig in \cite{Rosenzweig low reg},
which asks for a combined quasineutral mean field limit for Yudovich solutions, in the
spirit of \cite{Serfaty}. Such question already needs to be clarified for the quasineutral limit. Indeed, the  dissipative
solutions considered by Brenier \cite{Brenier} are not known to coincide with Yudovich solutions. In fact, the question of uniqueness for dissipative solutions with low regularity vorticity  is subtle, and is still not fully understood,  
see 4.4 in \cite{Lions} for the precise regularity assumptions needed in order to guarantee uniqueness of dissipative solutions. Finally, we stress that the main result (\Cref{main thm}) holds for a short time interval. In order to be able to extend this result for arbitrary times we expect that we will need to have at least a logarithmic (in $\varepsilon$) bound on $\left\Vert \rho^{\varepsilon}(t,\cdot)\right\Vert _{\infty}$ - which is open.

\textbf{Outline of the paper:} In \Cref{sec:Preliminaries} we
review some preliminaries, and in particular we include reminders
about the general theory for Yudovich solutions to \eqref{Euler Equation Intro} as well as general harmonic analysis estimates. In
\Cref{Sec: Harmonic analysis}, we continue with presenting
slight refinements of classical estimates for the maximal function and BMO norms.
Section \ref{sec:The-gyrokinetic-limit} reflects the main novelty
of the present work. In the first part of this section we follow closely
the approach taken in \cite{Griffin-Pickering Iacobelli 2023} in order to establish quantitative
estimates on $\left\Vert \rho^{\varepsilon}(t,\cdot)\right\Vert _{\infty}$ in the case of Vlasov-Poisson.
In the second part, we apply these estimates in order to quantify
the Hardy norm of integrals of the form \eqref{eq:typical integral}
by means of the modulated energy. These estimates eventually lead
to the proof of our  main result, which shows that, for short times, the modulated
energy approach can be adapted for Yudovich solutions, where a priori the gradient of the velocity is just BMO.
In \Cref{Construction-of-initial data } we explicitly construct well prepared initial data, i.e. initial data which simultaneously
realizes the asymptotic vanishing of the modulated energy as well
as growth conditions (with respect to $\varepsilon$) on the initial
$L^{\infty}$ norm of the density. Finally, in the Appendix we quickly recall the well-posedness theory for \eqref{eq:Vlasov Equation} on the two-dimensional torus.
\section{\label{sec:Preliminaries}Preliminaries }

The aim of this section is to recall basic properties of Yudovich solutions, following mainly \cite{Bahouri et al} and \cite{Majda Bertozzi}. 

\subsection{The Coulomb potential on $\mathbb{T}^{2}$}

Unlike in the planar settings, when working on the torus, the Coulomb
potential does not admit a simple closed formula and therefore we
start by briefly recalling how to properly define the Coulomb potential
on the torus. We fix some general notations. Let $\mathbb{T}^{2}\coloneqq\left[-\frac{1}{2},\frac{1}{2}\right]^{2},$
and set 

\[
\mathcal{\mathscr{P}}(\mathbb{T}^{2})\coloneqq\left\{ f:\mathbb{R}^{2}\rightarrow\mathbb{R}\left|\forall\alpha\in\mathbb{Z}^{2},\forall x\in\mathbb{\mathbb{R}}^{2}:f(x+\alpha)=f(x)\right.\right\}, 
\]

\[
\mathcal{\mathscr{P}}(\mathbb{T}^{2}\times\mathbb{R}^{2})\coloneqq\left\{ f:\mathbb{R}^{2}\times\mathbb{R}^{2}\rightarrow\mathbb{R}\left|\forall\alpha\in\mathbb{Z}^{2},\forall x,\xi\in\mathbb{\mathbb{R}}^{2}:f(x+\alpha,\xi)=f(x,\xi)\right.\right\} .
\]
For each non-negative integer $n\in\mathbb{N}_{0}$ define 

\[
C^{n,\alpha}(\mathbb{T}^{2})\coloneqq C^{n,\alpha}(\mathbb{R}^{2})\cap\mathcal{\mathscr{P}}(\mathbb{T}^{2}),\ C^{n,\alpha}(\mathbb{T}^{2}\times\mathbb{R}^{2})\coloneqq C^{n,\alpha}(\mathbb{R}^{2}\times\mathbb{R}^{2})\cap\mathcal{\mathscr{P}}(\mathbb{T}^{2}\times\mathbb{R}^{2}).
\]

\begin{lem}
\label{green function on torus -1} There exists a unique function
$V\in C^{\infty}(\mathbb{R}^{2}\setminus\mathbb{Z}^{2})\cap\mathcal{\mathscr{P}}(\mathbb{T}^{2})$
with the following properties. 

(a) For any $\varrho\in\mathcal{\mathscr{P}}(\mathbb{T}^{2})\cap L^{\infty}(\mathbb{T}^{2})$ with
$\int_{\mathbb{T}^{2}}\varrho(x)dx=0$ the function defined by 

\[
\Phi(x)\coloneqq\int_{\mathbb{T}^{2}}V(x-y)\varrho(y)dy
\]

is $C^{1}(\mathbb{T}^{2})$ and is the unique solution of 
\[
-\Delta\Phi=\varrho,\ \int_{\mathbb{T}^{2}}\Phi(x)dx=0.
\]
(b) There is a function $V_{0}\in C^{\infty}((\mathbb{R}^{2}\setminus\mathbb{Z}^{2})\cup\left\{ \mathbf{0}\right\} )$
such that for all $x\in\mathbb{R}^{2}\setminus\mathbb{Z}^{2}$ it
holds that 
\[
V(x)=V_{0}(x)+V_{1}(x),
\]
where $V_{1}(x)\coloneqq-\frac{1}{2\pi}\log(\left|x\right|)$. The
function $V$ is henceforth referred to as the Coulomb potential. 
\end{lem}

\subsection{The Yudovich Theorem }

By a weak solution to \eqref{Euler Equation Intro} we mean
the following. 
\begin{defn}
\label{Definition of weak solution } We say that $\omega:[0,T]\times\mathbb{T}^{2}\rightarrow\mathbb{R}$
is a \textit{weak solution to \eqref{Euler Equation Intro}} if $\omega\in L^{\infty}([0,T];L^{1}(\mathbb{T}^{2})\cap L^{p}(\mathbb{T}^{2}))$
for some $p>2$ and for all $\varphi\in W^{1,\infty}([0,T];C_{0}^{1}(\mathbb{R}^{2}))$
and all $t\in[0,T]$ it holds that 
\begin{align*}
\int_{\mathbb{T}^{2}}\omega(t,x)\varphi(t,x)dx=\int_{\mathbb{T}^{2}}\omega_{0}(x)\varphi(0,x)dx+\int_{0}^{t}\int_{\mathbb{T}^{2}}\omega(s,x)\partial_{s}\varphi(s,x)dxds\\ +\int_{0}^{t}\int_{\mathbb{T}^{2}}\nabla\varphi(s,x)\cdot\nabla^{\bot}V\star\omega(s,x)\omega(s,x)dxds.  
\end{align*}
\end{defn}
In \cite{Yudovich}, Yudovich obtained existence
and uniqueness of such a weak solution to \eqref{Euler Equation Intro} under the assumption that the initial vorticity is bounded. The periodic version of Yudovich's theorem stated here is taken from \cite{Crippa  Stefani}.  
\begin{thm}
\label{Yudovich theorem } Let $\omega_{0}\in L^{\infty}(\mathbb{T}^{2})$.
Then there exists a unique solution $\omega\in L^{\infty}([0,\infty); L^{\infty}(\mathbb{T}^{2})),$
in the sense of \Cref{Definition of weak solution }. Moreover
the $L^{p}$ norms of this solution are conserved, i.e. for all $t\in[0,\infty)$ it holds that  

\[
\left\Vert \omega(t,\cdot)\right\Vert _{p}=\left\Vert \omega_0(\cdot)\right\Vert _{p}, \ 1\leq p\leq \infty.
\]
\end{thm}
\begin{rem}
It is essential to keep in mind that the boundedness of the vorticity $\omega$
by no means guarantees the boundedness of $\nabla u$, due to the failure
of the Calderon-Zygmund inequality at the end-points exponents $p=1,\infty$
(see \cite[footnote 3, page 294]{Bahouri et al} and reference therein for details).
If the vorticity enjoys some Besov regularity than it can be shown
that the associated velocity field is Lipschitz (Proposition 7.7 in
\cite{Bahouri et al}), but the point is that we avoid to impose any
assumption of this sort. 
\end{rem}
Before going further we recall the definition of the function space
$\mathrm{BMO}(\Omega)$.
\begin{defn}
Let $\Omega=\mathbb{T}^{d}$ or $\Omega=\mathbb{R}^{d}$ and let $f\in L_{\mathrm{loc}}^{1}(\Omega)$.
We shall say that $f$ is of \textit{bounded mean oscillation} on
$\Omega$ or in brief $\mathrm{BMO}(\Omega)$, if 
\begin{equation}
\left\Vert f\right\Vert _{\mathrm{BMO}(\Omega)}\coloneqq\underset{Q\subset\Omega}{\sup}\frac{1}{\left|Q\right|}\int_{Q}\left|f(x)-\frac{1}{\left|Q\right|}\int_{Q}f\right|dx<\infty.\label{eq:-2-1}
\end{equation}
The supremum in Equation (\ref{eq:-2-1}) is taken over all cubes
$Q\subset\Omega$ with edges parallel to the axis, and $\left\Vert f\right\Vert _{\mathrm{\mathrm{BMO(}\Omega})}$
is called the $\mathrm{BMO}$ norm. 
\end{defn}
The failure of the boundedness of the Calderon-Zygmund type operators
from $L^{\infty}$ to itself is bypassed by the fact that operators
of the latter type are bounded from $L^{\infty}$ to $\mathrm{BMO}$, precisely
put we have the following:
\begin{thm}
\textup{\label{Endpoint CZ Theorem }}
Let $V$ be as in \Cref{green function on torus -1}. Let $T\omega=T_{kj}\omega\coloneqq\nabla_{k}\nabla_{j}^{\bot}V\star\omega$.
Then, there is a constant $A>0$ such that 

\[
\left\Vert T\omega\right\Vert _{\mathrm{BMO}(\mathbb{T}^{2})}\leq A\left\Vert \omega\right\Vert _{L^{\infty}(\mathbb{T}^{2})}
\]
for all $\omega\in L^{\infty}(\mathbb{T}^{2})$. In particular, if
$\omega$ is as in \Cref{Yudovich theorem } and $u$ is the
associated velocity field then 
\[
\underset{t\in[0,T]}{\sup}\left\Vert \nabla u(t,\cdot)\right\Vert _{\mathrm{\mathrm{BMO}(\mathbb{T}^{2})}}\lesssim\left\Vert \omega(0,\cdot)\right\Vert _{L^{\infty}(\mathbb{T}^{2})}.
\]
\end{thm}

\begin{proof}
Write $V=V_{0}+V_{1}$. Denote by $\varpi$ the  extension
of $\omega$ by $0$ to $\mathbb{R}^{2}$. Then 
\begin{equation}
\nabla_{k}\nabla_{j}^{\bot}V\star\omega(x)=\int_{\mathbb{R}^{2}}\nabla_{k}\nabla_{j}^{\bot}V_{1}(x-y)\mathbf{1}_{\mathbb{T}^{2}}(x-y)\varpi(y)dy+\int_{\mathbb{R}^{2}}\nabla_{k}\nabla_{j}^{\bot}V_{0}(x-y)\mathbf{1}_{\mathbb{T}^{2}}(x-y)\varpi(y)dy.\label{eq:-8}
\end{equation}
Let $\chi$ be a smooth radial compactly supported function such that
$\mathrm{supp}(\chi)\subset B_{\frac{1}{4}}(0)$ and $\chi\equiv1$
on $B_{\frac{1}{8}}(0)$. The first integral is 
\begin{align*}
\int_{\mathbb{R}^{2}}\nabla_{k}\nabla_{j}^{\bot}V_{1}(x-y)\mathbf{1}_{\mathbb{T}^{2}}(x-y)\varpi(y)dy=&\int_{\mathbb{R}^{2}}\nabla_{k}\nabla_{j}^{\bot}V_{1}(x-y)\chi(x-y)\mathbf{1}_{\mathbb{T}^{2}}(x-y)\varpi(y)dy\\
&+\int_{\mathbb{R}^{2}}\nabla_{k}\nabla_{j}^{\bot}V_{1}(x-y)(1-\chi(x-y))\mathbf{1}_{\mathbb{T}^{2}}(x-y)\varpi(y)dy\\
=&\int_{\mathbb{R}^{2}}\nabla_{k}\nabla_{j}^{\bot}V_{1}(x-y)\chi(x-y)\varpi(y)dy
\end{align*}
\begin{equation}
+\int_{\mathbb{R}^{2}}\nabla_{k}\nabla_{j}^{\bot}V_{1}(x-y)(1-\chi(x-y))\mathbf{1}_{\mathbb{T}^{2}}(x-y)\varpi(y)dy. \hspace{-4.0 cm} 
\label{eq:-7}
\end{equation}
Thanks to the usual Calderon-Zygmund theory on $\mathbb{R}^{2}$ (see e.g. \cite{Stein}) we
can estimate the $\mathrm{BMO}(\mathbb{T}^2)$ norm of first integral in (\ref{eq:-7}) as 
\begin{align*}
\left\Vert \int_{\mathbb{R}^{2}}\nabla_{k}\nabla_{j}^{\bot}V_{1}(\cdot-y)\chi(\cdot-y)\varpi(y)dy\right\Vert _{\mathrm{BMO}(\mathbb{T}^{2})}
&\leq\left\Vert \int_{\mathbb{R}^{2}}\nabla_{k}\nabla_{j}^{\bot}V_{1}(\cdot-y)\chi(\cdot-y)\varpi(y)dy\right\Vert _{\mathrm{BMO}(\mathbb{R}^{2})}\\&\leq C\left\Vert \varpi\right\Vert _{L^{\infty}(\mathbb{R}^{2})}=C\left\Vert \omega\right\Vert _{L^{\infty}(\mathbb{T}^{2})}.
\end{align*}
Observe the inequality $\left\Vert \cdot\right\Vert _{\mathrm{BMO}(\mathbb{T}^{2})}\leq2\left\Vert \cdot\right\Vert _{L^{\infty}(\mathbb{T}^{2})}$.
Therefore the first integral is 
\begin{align*}
 & \left\Vert  \int_{\mathbb{R}^{2}}\nabla_{k}\nabla_{j}^{\bot}V_{1}(\cdot-y)(1-\chi(\cdot-y))\mathbf{1}_{\mathbb{T}^{2}}(\cdot-y)\varpi(y)dy\right\Vert _{\mathrm{BMO}(\mathbb{T}^{2})}\\
\leq&2\left\Vert \int_{\mathbb{R}^{2}}\nabla_{k}\nabla_{j}^{\bot}V_{1}(\cdot-y)(1-\chi(\cdot-y))\mathbf{1}_{\mathbb{T}^{2}}(\cdot-y)\varpi(y)dy\right\Vert _{L^{\infty}(\mathbb{T}^{2})}\\
\leq&2\log(8)\left\Vert \varpi\right\Vert _{L^{\infty}(\mathbb{R}^{2})}=2\log(8)\left\Vert \omega\right\Vert _{L^{\infty}(\mathbb{T}^{2})}.
\end{align*}
As for the second integral in (\ref{eq:-8}), recall that by Lemma (\ref{green function on torus -1}) we have
$\nabla_{k}\nabla_{j}^{\bot}V_{0}\in L^{\infty}(\mathbb{T}^{2})$.
It follows that 
\begin{align*}
\left\Vert \int_{\mathbb{R}^{2}}\mathbf{1}_{\mathbb{T}^{2}}(\cdot-y)\nabla_{k}\nabla_{j}^{\bot}V_{0}(\cdot-y)\varpi(y)dy\right\Vert _{\mathrm{BMO}(\mathbb{T}^{2})}
&\leq2\left\Vert \int_{\mathbb{R}^{2}}\mathbf{1}_{\mathbb{T}^{2}}(\cdot-y)\nabla_{k}\nabla_{j}^{\bot}V_{0}(\cdot-y)\varpi(y)dy\right\Vert _{L^{\infty}(\mathbb{T}^{2})}\\&\leq2\left\Vert \nabla_{k}\nabla_{j}^{\bot}V_{0}\right\Vert _{L^{\infty}(\mathbb{T}^{2})}\left\Vert \omega\right\Vert _{L^{\infty}(\mathbb{T}^{2})}.
\end{align*}
\end{proof}

The following control on several time dependent quantities in terms
of the initial data are classical (and are also recalled in \cite{Rosenzweig low reg}
and references therein) 
\begin{lem}
\label{Estimate in terms of initial data } It holds that $\left\Vert u\right\Vert _{L^{\infty}([0,T];L^{\infty}(\mathbb{T}^{2}))}\lesssim\left\Vert \omega_{0}\right\Vert _{L^{\infty}(\mathbb{T}^{2})}$.
\end{lem}
\textit{Proof}. We have
\begin{align*}
\left\Vert \nabla^{\bot}V\star\omega(t,\cdot)\right\Vert _{L^{\infty}(\mathbb{T}^{2})}\leq\left\Vert \nabla^{\bot}V_{1}\star\omega(t,\cdot)\right\Vert _{L^{\infty}(\mathbb{T}^{2})}+\left\Vert \nabla^{\bot}V_{0}\star\omega(t,\cdot)\right\Vert _{L^{\infty}(\mathbb{T}^{2})}.    
\end{align*}
First 
\begin{align*}
\left\Vert \nabla^{\bot}V_{1}\star\omega(t,\cdot)\right\Vert _{L^{\infty}(\mathbb{T}^{2})}&\lesssim\left\Vert \frac{\boldsymbol{\mathbf{1}}_{\left|\cdot\right|\leq1}}{\left|\cdot\right|}\star\omega(t,\cdot)\right\Vert _{L^{\infty}(\mathbb{T}^{2})}+\left\Vert \frac{\boldsymbol{\mathbf{1}}_{\left|\cdot\right|\geq1}}{\left|\cdot\right|}\star\omega(t,\cdot)\right\Vert _{L^{\infty}(\mathbb{T}^{2})}\\&\lesssim\left\Vert \omega(t,\cdot)\right\Vert _{L^{\infty}\cap L^{1}(\mathbb{T}^{2})}=\left\Vert \omega_{0}\right\Vert _{L^{\infty}\cap L^{1}(\mathbb{T}^{2})}.    
\end{align*}
Second 
\[
\left\Vert \nabla^{\bot}V_{0}\star\omega(t,\cdot)\right\Vert _{L^{\infty}(\mathbb{T}^{2})}\leq\left\Vert \nabla^{\bot}V_{0}\right\Vert _{L^{\infty}(\mathbb{T}^{2})}\left\Vert \omega_{0}\right\Vert _{ L^1(\mathbb{T}^{2})}.
\]

\begin{flushright}
$\square$
\par\end{flushright}

\section{Harmonic Analysis} \label{Sec: Harmonic analysis}

\subsection{The maximal function and local BMO spaces}

Since we are working on a compact subset of $\mathbb{R}^{d}$,
we need to introduce a local variant of the $\mathrm{BMO}$ space. These spaces were studied
in \cite{Jonsson}, following the pioneering work of Fefferman-Stein on the topic. 

\textbf{Assumption $\mathbf{D}$}. We assume the following on the
pair $(\Omega,\mu)$. 

1. $\Omega\subset\mathbb{R}^{d}$ is non-empty and compact. 

2. $\mu$ is a positive Borel measure, finite on compact sets and
$\mathrm{supp}(\mu)\subset\Omega$. 

3. There are some constants $c,C$ such that for all $x\in\Omega$
and $0<r\leq1$ it holds that 

\[
cr^{d}\leq\mu(B(x,r))\leq Cr^{d}.
\]

\begin{defn}
Let assumption D hold. We say that $f\in L^{1}(\Omega,d\mu)$ is of
local bounded mean oscillation on $\Omega$ if 
\[
\left\Vert f\right\Vert _{\mathrm{bmo}(\Omega,\mu)}\coloneqq\underset{B=B(x,r),0<r\leq1,x\in\Omega}{\sup}\underset{c}{\inf}\frac{1}{\mu(B)}\int_{B}\left|f-c\right|d\mu+\underset{B=B(x,1),x\in\Omega}{\sup}\frac{1}{\mu(B)}\int_{B}\left|f\right|d\mu<\infty.
\]
The space $\mathrm{bmo}(\Omega,\mu)$ consists of all $f\in L^{1}(\Omega,d\mu)$
such that $\left\Vert f\right\Vert _{\mathrm{bmo}(\Omega,\mu)}<\infty$ and $\left\Vert f\right\Vert _{\mathrm{\mathrm{bmo(}\Omega})}$ is called the $\mathrm{bmo}$ norm. 
\end{defn}
We now define the local maximal function. 
\begin{defn}
Let 
\[
S\coloneqq\left\{ \varphi\in C_{0}^{\infty}(\mathbb{R}^{d})\left|\mathrm{supp}(\varphi)\subset B(0,1),\max\left\{ \left\Vert \varphi\right\Vert _{\infty},\left\Vert \nabla\varphi\right\Vert _{\infty}\right\} \leq1\right.\right\} .
\]
Given $f\in L^{1}(\mathbb{R}^{d},d\mu)$ we define the local maximal
function by 

\[
\mathcal{M}f(x)\coloneqq\underset{\varphi\in S} 
{\sup}\underset{0<r\leq1}{\sup}\left|\frac{1}{\left|B(x,r)\right|}\int_{\mathbb{R}^{d}}\varphi\left(\frac{x-y}{r}\right)f(y)d\mu(y)\right|,\ x\in\Omega.
\]
\end{defn}
Equipped with the above definitions we can now formulate the local
duality inequality 
\begin{thm} [\cite{Jonsson}, Theorem 4.2]
\label{local duality } Let assumption D hold. Then, it holds that 

\[
\left|\int_{\Omega}f(x)g(x)d\mu(x)\right|\leq C\left\Vert \mathcal{M}f\right\Vert _{L^{1}(\mu,\Omega)}\left\Vert g\right\Vert _{\mathrm{bmo}(\Omega,\mu)}.
\]
\end{thm}
In what follows, we take $\Omega=\left[-\frac{1}{2},\frac{1}{2}\right]^{2}$ and $\mu$ is defined
for any Borel set $B\subset\mathbb{R}^{2}$ by 

\[
\mu(B)\coloneqq\left|\Omega\cap B\right|.
\]
In order to apply the above result freely for the pair $(\Omega,\mu)$,
we verify that this pair satisfies assumption D. 
\begin{lem}
\label{assumption D is satisfied } The pair $(\Omega,\mu)$
satisfies assumption D. 
\end{lem}
\begin{proof}
Points 1 and 2 are trivial. The right hand side of the inequality
in point 3 is also clear. Let us prove the left hand side inequality.
Fix $x\in \Omega$ and $0<r\leq1$. Given a ball $B(x,r)$
with $x\in \Omega$ and $0<r\leq1$ let $Q(x,r)$ be the cube
prescribing $B(x,r)$, i.e. the cube with center $x$ and edge length
$2r$. Note that the ratio $\frac{\left|Q\right|}{\left|B\right|}$
is independent of $x,r$, and therefore it suffices to prove the inequality
for cubes with center $x$ and edge length $2r$. Given such a cube
$Q(x,r)$, consider the case where $Q\cap\mathbb{R}^{2}\setminus\Omega\neq\emptyset$
(otherwise there is nothing to prove). Note that in this case $\partial Q$
intersects $\partial\Omega$ in exactly 2 points, say $z,z'$
. If $z,z'$ are both on the same line then the upper\textbackslash lower
or right\textbackslash left half of $Q$ is contained in $\Omega$,
so that 
\[
\left|\Omega\cap Q\right|\geq\frac{1}{2}\left|Q\right|=2r^{2}.
\]
If $z'$ is on the line orthogonal to the line of $z$ then a quarter
of $Q$ is contained in $\Omega$ so that 
\[
\left|\Omega \cap Q\right|\geq\frac{1}{4}\left|Q\right|=\frac{r^{2}}{2}.
\]
Finally, consider the case where $z'$ is on the line parallel to
the line of $z$. In this case it must be that $\Omega\cap Q$
contains a rectangle with edges of lengths $r$ and $1$. Thus
\[
\left|\Omega\cap Q\right|\geq r\cdot1\geq r^{2}.
\]
 
\end{proof}
Next, we notice that $\left\Vert \cdot\right\Vert _{\mathrm{bmo}(\Omega,\mu)}$
is dominated by $\left\Vert \cdot\right\Vert _{\mathrm{BMO}(\mathbb{T}^{2})}$
as observed in the following Lemma. 
\begin{lem}
\label{ bmo vs BMO} It holds that 
\[
\left\Vert f\right\Vert _{\mathrm{bmo}(\Omega,\mu)}\leq\left\Vert f\right\Vert _{\mathrm{BMO}(\mathbb{T}^{2})}.
\]
\end{lem}
\begin{proof}
Given a ball $B(x,r)$ with $x\in \Omega$ and $0<r\leq1$ let  $Q(x,r)\subset \Omega$ be the cube prescribing
$B(x,r)$, i.e. the cube with center $x$ and edge length $2r$. As
before, consider the where case $\partial Q$ intersects $\partial \Omega$
in exactly 2 points, say $z,z'$ . If $z,z'$ are both on the same
line, or orthogonal lines, then enlarge the rectangle $Q\cap \Omega$
to a cube $Q'(x,r)$ with edge length $\leq2r$ and $Q\cap \Omega\subset Q'(x,r)$.
If $z,z'$ are on parallel lines then $r>\frac{1}{2}$ in which case
take $Q'(x,r)=\Omega$. With this choice of $Q'(x,r)$ we
see that 
\[
cr^{2}\leq\left|Q'(x,r)\right|\leq Cr^{2},\ B(x,r)\cap\Omega\subset Q'(x,r).
\]
Therefore 
\begin{align*}
\underset{c}{\inf}\frac{1}{\mu(B)}\int_{B}\left|f-c\right|d\mu&\leq\frac{1}{\mu(B)}\int_{B}\left|f-\frac{1}{\left|Q'\right|}\int_{Q'}f\right|d\mu\\ 
&\leq\frac{1}{\mu(B)}\int_{Q'}\left|f-\frac{1}{\left|Q'\right|}\int_{Q'}f\right|dx\lesssim\frac{1}{\left|Q'\right|}\int_{Q'}\left|f-\frac{1}{\left|Q'\right|}\int_{Q'}f\right|dx,
\end{align*}
which shows that 

\[
\underset{B(x,r),0<r\leq1,x\in\Omega}{\sup}\underset{c}{\inf}\frac{1}{\mu(B)}\int_{B}\left|f-c\right|d\mu\lesssim\underset{Q\subset \Omega}{\sup}\frac{1}{\left|Q\right|}\int_{Q}\left|f-\frac{1}{\left|Q\right|}\int_{Q}f\right|dx\leq\left\Vert f\right\Vert _{\mathrm{BMO}(\mathbb{T}^{2})}.
\]
Since 
\[
\frac{1}{\mu(B(x,1))}\int_{B(x,1)}\left|f\right|d\mu\lesssim\int_{\mathbb{T}^{2}}\left|f\right|dx\lesssim\left\Vert f\right\Vert _{\mathrm{BMO}(\mathbb{T}^{2})},
\]
we have thus proved 
\[
\left\Vert f\right\Vert _{\mathrm{bmo}(\Omega,\mu)}\lesssim_{d}\left\Vert f\right\Vert _{\mathrm{BMO}(\mathbb{T}^{2})}.
\]
\end{proof}
The combination of Theorem \ref{local duality }, Lemma \ref{assumption D is satisfied }
and Lemma \ref{ bmo vs BMO} yields 
\begin{cor} \label{Duality inequality on torus}
It holds that 
\end{cor}
\[
\left|\int_{\mathbb{T}^{2}}f(x)g(x)dx\right|\leq C\left\Vert \mathcal{M}f\right\Vert _{L^{1}(\mathbb{T}^{2})}\left\Vert g\right\Vert _{\mathrm{BMO}(\mathbb{T}^{2})}.
\]

\subsection{A Wiener type inequality}
The next Lemma is a Wiener type inequality, and its proof follows
closely \cite{Zygmund}. 
\begin{lem}
\label{Orlicz space type estimate -1-1} Suppose that $g\in L^{1}(\mathbb{T}^{d})$.
Then, for each $\eta>0$ it holds that 

\[
\int_{\mathbb{T}^{d}}\mathcal{M}g(x)dx\lesssim_{d}\eta+\int_{\mathbb{T}^{d}}\left|g\right|(x)\log_{+}\left|g\right|(x)dx+\log(\frac{1}{\eta})\left\Vert g\right\Vert _{L^{1}(\mathbb{T}^{d})}.
\]
\end{lem}
\begin{proof}
We have
\begin{align*}
\int_{\mathbb{T}^{d}}\mathcal{M}g(x)dx&=\int_{0}^{\infty}\left|\left\{ x\in\mathbb{T}^{d}:\mathcal{M}g>t\right\} \right|dt\\
&=\frac{1}{2}\int_{0}^{\eta}\left|\left\{ x\in\mathbb{T}^{d}:\mathcal{M}g>2t\right\} \right|dt+\frac{1}{2}\int_{\eta}^{\infty}\left|\left\{ x\in\mathbb{T}^{d}:\mathcal{M}g>2t\right\} \right|dt.
\end{align*}

Evidently $\left|\left\{ x\in\mathbb{T}^{d}:\mathcal{M}g>2t\right\} \right|\leq1$,
so that 
\[
\frac{1}{2}\int_{0}^{\eta}\left|\left\{ x\in\mathbb{T}^{d}:\mathcal{M}g>t\right\} \right|dt\leq\eta.
\]
We move to handle the second integral.\textbf{ }Let us estimate the
second integral. We split $g$ to the parts in which it is smaller
and larger than $g$, i.e. $g=g_{1}+g_{2}$ where $g_{1}\coloneqq g\mathbf{1}_{\left\{ \left|g\right|>t\right\} }$
and $g_{2}=g\mathbf{1}_{\left\{ \left|g\right|\leq t\right\} }$.
Evidently $\mathcal{M}g_{2}\leq t$ so that 
\[
\left\{ x\in\mathbb{T}^{d}\left|\mathcal{M}g>2t\right.\right\} \subset\left\{ x\in\mathbb{T}^{d}\left|\mathcal{M}g_{1}>t\right.\right\} .
\]
Thus, by the Hardy-Littlewood maximal Inequality we have (Theorem
7.4 in \cite{Rudin}) 

\begin{align*}
\int_{\eta}^{\infty}\left|\left\{ x\in\mathbb{T}^{d}:\mathcal{M}g(x)>2t\right\} \right|dt&\leq\int_{\eta}^{\infty}\left|\left\{ x\in\mathbb{T}^{d}:\mathcal{M}g_{1}(x)>t\right\} \right|dt\lesssim\int_{\eta}^{\infty}\frac{1}{t}\left(\int_{\left|g\right|>t}\left|g\right|(x)dx\right)dt\\
&=\int_{\eta}^{\infty}\int_{\mathbb{T}^{d}}\frac{1}{t}\mathbf{1}_{\left\{ \left|g\right|>t\right\} }\left|g\right|(x)dxdt=\int_{\mathbb{T}^{d}}\left|g\right|(x)\int_{\eta}^{\infty}\frac{1}{t}\mathbf{1}_{\left\{ \left|g\right|>t\right\} }dtdx\\
&\leq\int_{\mathbb{T}^{d}}\left|g\right|\int_{\eta}^{\max\left\{ \left|g\right|,1\right\} }\frac{dt}{t}dx=\int_{\mathbb{T}^{d}}\left|g\right|\log_{+}(\left|g\right|)dx-\log(\eta)\left\Vert g\right\Vert _{L^{1}(\mathbb{T}^{d})}.
\end{align*}
\end{proof}

\section{\label{sec:The-gyrokinetic-limit} The quasineutral limit for Yudovich solutions }

\subsection{
Propagation in time of $\left\Vert \rho^{\varepsilon}(t,\cdot)\right\Vert _{\infty}$ and partial velocity moments}

We first recall a classical velocity moment inequality for bounded distribution functions and the conservation of energy for the Vlasov--Poisson system, which we take from \cite{Griffin-Pickering Iacobelli 2023}.
\begin{lem}
\textup{\label{interpolation inequality } (Lemma 2.3 in \cite{Griffin-Pickering Iacobelli 2023})}
Let $d\geq1$. Let $0\leq f\in L^{\infty}(\mathbb{T}^{d}\times\mathbb{R}^{d})$
satisfy for some $k>1$ 

\[
M_{k}\coloneqq\int_{\mathbb{T}^{d}\times\mathbb{R}^{d}}\left|\xi\right|^{k}f(x,\xi)dxd\xi<\infty.
\]
Then the spatial density 

\[
\rho(x)\coloneqq\int_{\mathbb{R}^{d}}f(x,\xi)d\xi
\]
belongs to $L^{1+\frac{k}{d}}(\mathbb{T}^{d})$ with the estimate
\[
\left\Vert \rho\right\Vert _{1+\frac{k}{d}}\leq C_{k,d}\left\Vert f\right\Vert _{\infty}^{\frac{k}{d+k}}M_{k}^{\frac{d}{d+k}}.
\]
\end{lem}
\begin{lem}
\label{Conservation of energy } Let the assumptions of \Cref{thm:(Existence-for-Vlasov--Poisson)} hold
and let $f^{\varepsilon}(t,x,\xi)$ be the solution to the Cauchy
problem
\[
\left\{ \begin{array}{lc}
\partial_{t}f^{\varepsilon}+\xi\cdot\nabla_{x}f^{\varepsilon}-\nabla\Phi^{\varepsilon}\cdot\nabla_{\xi}f^{\varepsilon}=0,\ f^{\varepsilon}(0,x,\xi)=f_{0}^{\varepsilon}\\
\rho^{\varepsilon}-1=-\varepsilon\Delta\Phi^{\varepsilon},\ \rho^{\varepsilon}(t,x)=\int_{\mathbb{R}^{2}}f^{\varepsilon}(t,x,\xi)d\xi.
\end{array}\right.
\]
Then the energy 
\[
\mathcal{E}_{\varepsilon}^{0}(t)\coloneqq\frac{1}{2}\int_{\mathbb{T}^{2}\times\mathbb{R}^{2}}\left|\xi\right|^{2}f^{\varepsilon}(t,x,\xi)dxd\xi+\frac{\varepsilon}{2}\int_{\mathbb{T}^{2}}\left|\nabla\Phi^{\varepsilon}\right|^{2}(t,x)dx
\]
is conserved, i.e. $\mathcal{E}_{\varepsilon}^{0}(t)=\mathcal{E}_{\varepsilon}^{0}(0)$. 
\end{lem}
The combination of \Cref{interpolation inequality } and \Cref{Conservation of energy } gives a uniform estimate in time and
$\varepsilon$ on $\left\Vert \rho^{\varepsilon}(t,\cdot)\right\Vert _{2}$,
i.e. 
\begin{lem}
\label{Estimate on L2 norm of density } Let the assumptions $\mathbf{H1}$-$\mathbf{H2}$ hold. Suppose that $\int_{\mathbb{R}^{2}}\left|\xi\right|^{2}f_{0}^{\varepsilon}(x,\xi)d\xi\leq M_{2}$.
Then 
\[
\left\Vert \rho^{\varepsilon}(t,\cdot)\right\Vert _{2}\leq C\sqrt{M_{2}}\sqrt{\overline{M}_{\varepsilon}},
\]
where $\overline{M}_{\varepsilon}$ is from $\mathbf{H2}$.
\end{lem}
The main purpose of this section is to study the propagation in time
of the $L^{\infty}$ norm of the density and partial velocity moments of a given a solution to the Vlasov-Poisson equation \eqref{eq:Vlasov Equation}. The
fact that we are able to get \emph{polynomial} explicit bounds in terms of $\varepsilon$
will prove to be crucial for closing the estimate for the modulated
energy. Our approach follows closely that of \cite[section 5]{Griffin-Pickering Iacobelli 2023}, adapted to the case of electrons.
Define the time dependent quantity 
\[
Q_{\ast}(t)\coloneqq\underset{(x,\xi)}{\sup}\left|\xi-\Xi_{\left(t,x,\xi\right)}(0)\right|,
\]
where $\Xi_{\left(t,x,\xi\right)}$ is the velocity characteristic of the Vlasov--Poisson system defined just below at \eqref{eq:-3-1-1-1}.
The first step is to obtain a bound on $\left\Vert \rho^{\varepsilon}(t,\cdot)\right\Vert _{\infty}$ and partial velocity moments 
by means of $Q_{\ast}(t)$, which is the content of the following
lemma. 

\begin{lem}
\label{bound on density } Let assumptions $\mathbf{H1}$-$\mathbf{H2}$ hold. Let $f^{\varepsilon}(t,x,\xi)$ be a solution to
the Cauchy problem 

\[
\left\{ \begin{array}{lc}
\partial_{t}f^{\varepsilon}+\xi\cdot\nabla_{x}f^{\varepsilon}-\nabla\Phi^{\varepsilon}\cdot\nabla_{\xi}f^{\varepsilon}=0,\ f^{\varepsilon}(0,x,\xi)=f_{0}^{\varepsilon} \vspace{0.25 cm} \\
\rho^{\varepsilon}-1=-\varepsilon\Delta\Phi^{\varepsilon},\ \rho^{\varepsilon}(t,x)=\int_{\mathbb{R}^{2}}f^{\varepsilon}(t,x,\xi)d\xi.
\end{array}\right.
\]
Given $0\leq k\leq 2$ let $m_{k}^{\varepsilon}$ designate the $k$-th partial velocity moment, i.e.  
\begin{align*}
m_{k}^{\varepsilon}(t,x)\coloneqq \int_{\mathbb{R}^{2}} \left\vert \xi \right\vert^{k}f^{\varepsilon}(t,x,\xi)d\xi.    
\end{align*} Then 
\[
\left\Vert m_{k}^{\varepsilon}(t,\cdot)\right\Vert _{\infty}\leq C(1+Q_{\ast}^{k+2}(t))
\]
where 
\[
C=C(\overline{M}_\varepsilon,k_{0},k)=\overline{M}_{\varepsilon}c_{k_{0},k}.
\]
In particular, 
\begin{align}
\left\Vert \rho^{\varepsilon}(t,\cdot) \right\Vert_{\infty}\leq    C\left(1+Q_{\ast}^{2}(t)\right) \label{density est}
\end{align}
and 
\begin{align}
\left\Vert m^{\varepsilon}_{2}(t,\cdot) \right\Vert_{\infty}\leq    C\left(1+Q_{\ast}^{4}(t)\right) \label{partial velocity est}
\end{align}
where $C=C(\overline{M}_\varepsilon,k_{0})=\overline{M}_\varepsilon c_{k_{0}}$. 
\end{lem}
\textit{Proof}. Consider the trajectories 
\begin{equation}
\left\{ \begin{array}{lc}
\frac{d}{ds}X_{\left(t,x,\xi\right)}(s)=\Xi_{\left(t,x,\xi\right)}(s),\ X_{\left(t,x,\xi\right)}(t)=x \vspace{0.25 cm}\\
\frac{d}{ds}\Xi_{\left(t,x,\xi\right)}(s)=\nabla\Phi_{\left(t,x,\xi\right)}^{\varepsilon}(s,X(s)),\ \Xi_{\left(t,x,\xi\right)}(t)=\xi.
\end{array}\right.\label{eq:-3-1-1-1}
\end{equation}
We can write the solution $f^{\varepsilon}(t,x,\xi)$ in a Lagrangian
manner, namely 

\[
f^{\varepsilon}(t,x,\xi)=f_{0}^{\varepsilon}(X_{\left(t,x,\xi\right)}(0),\Xi_{\left(t,x,\xi\right)}(0))
\]
for a.e. $(x,\xi)\in\mathbb{T}^{2}\times\mathbb{R}^{2}$. Therefore,
by assumption we have the estimate 

\[
\left\vert \xi \right\vert^{k}f^{\varepsilon}(t,x,\xi)\leq\frac{\left\vert \xi \right\vert^{k}\overline{M}_{\varepsilon}}{1+\left|\Xi_{\left(t,x,\xi\right)}(0)\right|^{k_{0}}},\ (t,x,\xi)\in[0,T]\times\mathbb{T}^{2}\times\mathbb{R}^{2}.
\]
Furthermore 
\[
\left|\Xi_{\left(t,x,\xi\right)}(0)\right|\geq\left|\left|\xi\right|-\left|\xi-\Xi_{\left(t,x,\xi\right)}(0)\right|\right|\geq\left(\left|\xi\right|-\left|\xi-\Xi_{\left(t,x,\xi\right)}(0)\right|\right)_{+}\geq\left(\left|\xi\right|-Q_{\ast}(t)\right)_{+},
\]
and as a result we get 

\[
\left\vert \xi \right\vert^{k}f^{\varepsilon}(t,x,\xi)\leq\frac{\left\vert \xi \right\vert^{k}\overline{M}_{\varepsilon}}{1+\left(\left|\xi\right|-Q_{\ast}(t)\right)_{+}^{k_{0}}}.
\]
Integration in $\xi$ of the last inequality yields the estimate 
\begin{align}
m_{k}^{\varepsilon}(t,x)\leq\overline{M}_{\varepsilon}\int_{\mathbb{R}^{2}}\frac{\left|\xi\right|^{k}}{1+\left(\left|\xi\right|-Q_{\ast}(t)\right)_{+}^{k_{0}}}d\xi.
\label{part moment ine}
\end{align}
Noticing that the integrand in the last integral is radial in $\xi$,
we can write 

\[
\int_{\mathbb{R}^{2}}\frac{\left|\xi\right|^{k}}{1+\left(\left|\xi\right|-Q_{\ast}(t)\right)_{+}^{k_{0}}}d\xi=\int_{0}^{\infty}\frac{r^{k+1}}{1+\left(r-Q_{\ast}(t)\right)_{+}^{k_{0}}}dr.
\]
The estimate for the last integral proceeds as follows. First we split
it as 

\begin{equation}
\int_{0}^{Q_{\ast}(t)}r^{k+1}dr+\int_{Q_{\ast}(t)}^{\infty}\frac{r^{k+1}}{1+\left(r-Q_{\ast}(t)\right)_{+}^{k_{0}}}dr=\frac{1}{k+2}Q_{\ast}^{k+2}(t)+\int_{0}^{\infty}\frac{\left(r+Q_{\ast}(t)\right)^{k+1}}{1+r^{k_{0}}}dr.\label{eq:-6-1}
\end{equation}
Since $k_{0}>4$ and $0\leq k\leq 2$, the second integral in the right hand side of Equation
\eqref{eq:-6-1} is bounded by
\[
\int_{0}^{\infty}\frac{(r+Q_{\ast}(t))^{k+1}} {1+r^{k_{0}}}dr\leq
2^{k+1}\int_{0}^{\infty}\frac{r^{k+1}} {1+r^{k_{0}}}dr+ 2^{k+1}Q_{\ast}^{k+1}(t)\int_{0}^{\infty} \frac{1}{1+r^{k_{0}}}dr\leq
c'_{k_{0},k}(1+Q_{\ast}^{k+1}(t)).
\]
To conclude, \eqref{part moment ine} and the last inequality shows that 
\[
m^{\varepsilon}_{k}(t,x)\leq\frac{\overline{M}_{\varepsilon}}{k+2}Q_{\ast}^{k+2}(t)+\overline{M}_{\varepsilon}c'_{k_{0},k}(1+Q_{\ast}^{k+1}(t))\leq C(\overline{M}_{\varepsilon},k_{0},k)(1+Q_{\ast}^{k+2}(t)).
\]
Invoking the last inequality with $k=0$ and $k=2$ yields \eqref{density est} and \eqref{partial velocity est} respectively. 
\qed
\\

The second step is the following functional inequality which ultimately
leads to an estimate on $\left\Vert \nabla\Phi^{\varepsilon}\right\Vert _{L^{\infty}(\mathbb{T}^{2})}$. 

\begin{lem}
\label{ESTIMATE ON FORCE FIELD } Suppose that $\rho^\varepsilon\in L^{\infty}(\mathbb{T}^{2})$ and $\left\Vert \rho^\varepsilon\right\Vert _{L^{1}(\mathbb{T}^{2})}=1$.
Denote
\[
\Phi^{\varepsilon}(x)\coloneqq\frac{1}{\varepsilon}(-\Delta)^{-1}\left(\rho^\varepsilon-1\right).
\]
Then, for all $l\in [0,1[$ it holds that 
\[
\left\Vert \nabla\Phi^{\varepsilon}\right\Vert _{L^{\infty}(\mathbb{T}^{2})}\leq\frac{2}{\varepsilon}\left\Vert \nabla V_{0}\right\Vert _{L^{\infty}(\mathbb{T}^{2})}+\frac{1}{\varepsilon}l\left(\left\Vert \rho^\varepsilon\right\Vert _{L^{\infty}(\mathbb{T}^{2})}+1\right)+\frac{1}{\varepsilon}\sqrt{\left|\log\left(l\right)\right|}\left(\left\Vert \rho^\varepsilon\right\Vert _{L^{2}(\mathbb{T}^{2})}+1\right).
\]
\end{lem}
\textit{Proof}.
\[
\nabla\Phi^{\varepsilon}=\frac{1}{\varepsilon}\int_{}\frac{x-y}{\left|x-y\right|^{2}}\left(\rho^\varepsilon-1\right)(t,y)dy+\frac{1}{\varepsilon}\nabla V_{0}\star\left(\rho^\varepsilon-1\right).
\]
By \Cref{green function on torus -1}, (b) we have $\nabla V_{0}\in C^{\infty}(\mathbb{T}^{2})$
and therefore we can estimate 
\begin{align*}
\left\Vert \nabla V_{0}\star\left(\rho^\varepsilon-1\right)\right\Vert _{L^{\infty}(\mathbb{T}^{2})}
\leq\left\Vert \nabla V_{0}\right\Vert _{L^{\infty}(\mathbb{T}^{2})}\left\Vert \rho^\varepsilon-1\right\Vert _{_{L^{1}(\mathbb{T}^{2})}}\leq2\left\Vert \nabla V_{0}\right\Vert _{L^{\infty}(\mathbb{T}^{2})}.
\end{align*}
As for the first integral we split it as 

\begin{equation}
=\frac{1}{\varepsilon}\int_{\left|x-y\right|\leq l}\frac{x-y}{\left|x-y\right|^{2}}\left(\rho^\varepsilon-1\right)(t,y)dy+\frac{1}{\varepsilon}\int_{\left|x-y\right|\geq l}\frac{x-y}{\left|x-y\right|^{2}}\left(\rho^\varepsilon-1\right)(t,y)dy.\label{eq:}
\end{equation}
For the first integral in \eqref{eq:}), note that 

\begin{equation}
\left|\int_{\left|x-y\right|\leq l}\frac{x-y}{\left|x-y\right|^{2}}\left(\rho^\varepsilon-1\right)(t,y)dy\right|\leq\left\Vert \frac{\mathbf{1}_{\left|\cdot\right|\leq l}}{\left|\cdot\right|}\star(\rho^\varepsilon-1)\right\Vert _{L^{\infty}(\mathbb{T}^{2})}\leq l\left\Vert \rho^\varepsilon-1\right\Vert _{L^{\infty}(\mathbb{T}^{2})}\leq l\left(\left\Vert \rho^\varepsilon\right\Vert _{L^{\infty}(\mathbb{T}^{2})}+1\right).\label{eq:-7-2}
\end{equation}
For the second integral, observe that 

\begin{align*}
\left\Vert \frac{\mathbf{1}_{\left|\cdot\right|\geq l}}{\left|\cdot\right|}\star\left(\rho^\varepsilon-1\right)\right\Vert _{L^{\infty}(\mathbb{T}^{2})}&\leq\left\Vert \frac{\mathbf{1}_{\left|\cdot\right|\geq l}}{\left|\cdot\right|}\right\Vert _{L^{2}(\mathbb{T}^{2})}\left\Vert \rho^\varepsilon\right\Vert _{L^{2}(\mathbb{T}^{2})}\\
&\leq\sqrt{\left|\log\left(l\right)\right|}\left(\left\Vert \rho^\varepsilon\right\Vert _{L^{2}(\mathbb{T}^{2})}+1\right).
\end{align*}
Gathering the inequalities gives 
\[
\left\Vert \nabla\Phi^{\varepsilon}\right\Vert _{L^{\infty}(\mathbb{T}^{2})}\leq\frac{2}{\varepsilon}\left\Vert \nabla V_{0}\right\Vert _{L^{\infty}(\mathbb{T}^{2})}+\frac{1}{\varepsilon}l\left(\left\Vert \rho^\varepsilon\right\Vert _{L^{\infty}(\mathbb{T}^{2})}+1\right)+\frac{1}{\varepsilon}\sqrt{\left|\log\left(l\right)\right|}\left(\left\Vert \rho^\varepsilon\right\Vert _{L^{2}(\mathbb{T}^{2})}+1\right).
\]

\begin{flushright}
$\square$
\par\end{flushright}
We are now in a good position to study quantitatively the propagation
in time of $\left\Vert \rho^{\varepsilon}(t,\cdot)\right\Vert _{\infty}$ and $\left\Vert m_{2}^{\varepsilon}(t,\cdot) \right\Vert_{\infty}$.  
\begin{lem}
\label{propagation of the density }
Let assumptions $\mathbf{H1}$-$\mathbf{H2}$ hold. Let $f^{\varepsilon}(t,x,\xi)$ be the solution to the Cauchy problem
\eqref{eq:Vlasov Equation} with initial data $f_{0}^{\varepsilon}$. Then, for
all $t\in[0,\infty)$ it holds that 

\[
\left\Vert \rho^{\varepsilon}(t,\cdot)\right\Vert _{\infty}\lesssim \Gamma_{\varepsilon}(t),
\]
and 
\[\left\Vert m_{2}^{\varepsilon}(t,\cdot)\right\Vert_{\infty}\lesssim \Gamma_{\varepsilon}^{2}(t)\]
where 
\begin{align*}
\Gamma_{\varepsilon}(t)\coloneqq \mathbf{M}_{\varepsilon}+\mathbf{M}_{\varepsilon}^{3}t^{2}\left(1+\log\left(1+\mathbf{M}_{\varepsilon}t\right)\right)   
\end{align*}
and where 
$\mathbf{M}_{\varepsilon}\coloneqq\frac{1+(\overline{M_{\varepsilon}})^{\frac{1}{2}}}{\varepsilon}$.

\end{lem}
\begin{proof}
We can write the solution $f^{\varepsilon}(t,x,\xi)$ in a Lagrangian
manner, namely 

\[
f^{\varepsilon}(t,x,\xi)=f_{0}^{\varepsilon}(X_{\left(t,x,\xi\right)}(0),\Xi_{\left(t,x,\xi\right)}(0))
\]
for a.e. $(x,\xi)\in\mathbb{T}^{2}\times\mathbb{R}^{2}$ where $(X_{\left(t,x,\xi\right)},\Xi_{\left(t,x,\xi\right)})$
are given by \eqref{eq:-3-1-1-1}. By \Cref{ESTIMATE ON FORCE FIELD }
applied with $ l=\frac{1}{1+\left\Vert \rho^{\varepsilon}(t,\cdot)\right\Vert _{\infty}}$, \Cref{Estimate on L2 norm of density }, and \Cref{bound on density } we have 
\begin{equation}
\left\Vert \nabla\Phi_{\varepsilon}(t,\cdot)\right\Vert _{L^{\infty}(\mathbb{T}^{2})}\lesssim\frac{1}{\varepsilon}+\frac{1+(\overline{M_{\varepsilon}})^{\frac{1}{2}}}{\varepsilon}\sqrt{\left|\log\left(\overline{M}_{\varepsilon}(1+Q_{\ast}^{2}(t))\right)\right|}.\label{eq:-3}
\end{equation}
Setting $\mathbf{M}_{\varepsilon}\coloneqq\frac{1+(\overline{M_{\varepsilon}})^{\frac{1}{2}}}{\varepsilon}$  we see that 
\[
\left|\Xi_{\left(s,x,\xi\right)}(0)-\xi\right|\leq\left|\int_{0}^{s}\nabla\Phi_{\left(t,x,\xi\right)}^{\varepsilon}(\tau,X(\tau))d\tau\right|
\]
so that in view of \eqref{eq:-3} we obtain 
\begin{align*}
\underset{s\in[0,t]}{\sup}Q_{\ast}(s)\leq\underset{s\in[0,t]}{\sup}\int_{0}^{s}\mathbf{M}_{\varepsilon}\left(1+\sqrt{\log(\mathbf{M}_{\varepsilon}(1+Q_{\ast}^{2}(\tau)))}\right)d\tau
\\\leq t\mathbf{M}_{\varepsilon}\left(1+\sqrt{\log\left(\mathbf{M}_{\varepsilon}\left(1+\underset{s\in[0,t]}{\sup}Q_{\ast}^{2}(s)\right)\right)}\right).    
\end{align*}
Solving the above inequality (see Lemma A.1 in \cite{Griffin-Pickering Iacobelli 2023} for
more details) yields 

\[
\underset{s\in[0,t]}{\sup}Q_{\ast}(s)\lesssim \mathbf{M}_{\varepsilon}t\sqrt{1+\log\left(1+\mathbf{M}_{\varepsilon}t\right)}.
\]
Thanks to \Cref{bound on density }

\[
\left\Vert \rho^{\varepsilon}(t,\cdot)\right\Vert _{\infty}\lesssim \overline{M}_{\varepsilon}(1+Q^{2}_{\ast}(t))\lesssim \mathbf{M}_{\varepsilon}+\mathbf{M}_{\varepsilon}^{3}t^{2}\left(1+\log\left(1+\mathbf{M}_{\varepsilon}t\right)\right)\coloneqq\Gamma_{\varepsilon}(t)
\]
and 
\begin{align*}
\left\Vert m_{2}^{\varepsilon}(t,\cdot)\right\Vert_{\infty}\lesssim \overline{M}_{\varepsilon}(1+Q^{4}_{\ast}(t))\lesssim \Gamma_{\varepsilon}^{2}(t).   \end{align*}
\end{proof}

\subsection{The Gr\"onwall estimate for the modulated energy}

We consider the modulated energy, which we recall is given by 
\begin{equation}
E_{\varepsilon}(t)\coloneqq\frac{1}{2}\int\left|\xi-u(t,x)\right|^{2}f^{\varepsilon}(t,x,\xi)dxd\xi+\frac{\varepsilon}{2}\int\left|\nabla\Phi^{\varepsilon}\right|^{2}(t,x)dx.\label{modulated energy}
\end{equation}
Here $f^{\varepsilon}(t,x,\xi)$ is a solution of the Vlasov--Poisson
system 

\begin{equation}
\left\{ \begin{array}{c}
\partial_{t}f^{\varepsilon}+\xi\cdot\nabla_{x}f^{\varepsilon}-\nabla\Phi^{\varepsilon}\cdot\nabla_{\xi}f^{\varepsilon}=0,\ f^{\varepsilon}(0,x,\xi)=f_{0}^{\varepsilon}\\
\rho^{\varepsilon}-1=-\varepsilon\Delta\Phi^{\varepsilon},\ \rho^{\varepsilon}(t,x)=\int_{\mathbb{R}^{2}}f^{\varepsilon}(t,x,\xi)d\xi,
\end{array}\right.\label{VP section 4}
\end{equation}
and $(\omega,u)$ is a solution to the 2D incompressible Euler equation 
\begin{equation}
\partial_{t}\omega+\mathrm{div}(\omega u)=0,\ \omega=\Delta\psi,\ u=-\nabla^{\bot}\psi.\label{Euler section 4}
\end{equation}
We introduce the notation 

\[
d(u)\coloneqq\frac{1}{2}\left(\nabla u+\nabla^{T}u\right)
\]
and 

\[
A(u)\coloneqq\partial_{t}u+u\cdot\nabla u.
\]
The following calculation can be found in Appendix 7 of \cite{Brenier}. 
\begin{thm}
Let assumptions \textbf{H1-H4} and let the assumption of \Cref{Yudovich theorem }
hold. Then, it holds that 
\begin{align*}
\dot{E_{\varepsilon}}(t)=-&\int_{\mathbb{T}^{2}}\int_{\mathbb{R}^{2}}d(u)(t,x):(\xi-u(t,x))\otimes (\xi-u(t,x))f^{\varepsilon}(t,x,\xi)dxd\xi \\
& +\varepsilon\int_{\mathbb{T}^{2}}d(u)(t,x):\nabla\Phi^{\varepsilon} \otimes \nabla\Phi^{\varepsilon}dx
+\int_{\mathbb{T}^{2}}A(u)(t,x)\cdot u(t,x)\left(\rho^{\varepsilon}(t,x)-1\right)dx\coloneqq\stackrel[k=1]{3}{\sum}I_{k}.
\end{align*}
\end{thm}
In order to conclude the Gr\"onwall estimate, we now aim to estimate
each of the summands $I_{k}$ by means of the modulated energy $E_{\varepsilon}(t)$.
The main theorem of this section, from which our main result \Cref{main thm} follows directly, is 
\begin{thm}
Let assumptions \textbf{H1-H4} hold. Let $f^{\varepsilon}(t,x,\xi)$ and $(\omega(t,x),u(t,x))$
be the solutions to \eqref{VP section 4} and \eqref{Euler section 4}
respectively (ensured by \Cref{thm:(Existence-for-Vlasov--Poisson)}
and \Cref{Yudovich theorem } respectively). Let $E_{\varepsilon}(t)$
be given by formula \eqref{modulated energy}. Then, there is some
$T_{\ast}=T_{\ast}(\left\Vert \omega_{0}\right\Vert _{L^{\infty}(\mathbb{T}^2)})$
such that $\underset{t\in[0,T_{\ast}]}{\sup}E_{\varepsilon}(t)\underset{\varepsilon\rightarrow0}{\rightarrow}0$. 
\end{thm}
\textit{Proof}. 
\textbf{Step 1. Estimate on $I_{1}$.} Let $F^{\varepsilon}:\mathbb{T}^{2}\rightarrow\mathbb{R}$
be defined as follows 
\[
F^{\varepsilon}(t,x)\coloneqq\int_{\mathbb{R}^{2}}(\xi-u(t,x))^{i}(\xi-u(t,x))^{j}f^{\varepsilon}(t,x,\xi)d\xi.
\]
First, using Einstein's summation 
\begin{align*}
\left|I_{1}\right|=&\left|\int_{\mathbb{T}^{2}\times\mathbb{R}^{2}}\partial_{x_{i}}u^{j}(t,\cdot)(\xi-u(t,\cdot))^{i}(\xi-u(t,\cdot))^{j}f^{\varepsilon}(t,x,\xi)dxd\xi\right|\\
&\leq\underset{t\in[0,T]}{\sup}\left\Vert \nabla u(t,\cdot)\right\Vert _{\mathrm{BMO(\mathbb{T}^{2})}}\left\Vert \mathcal{M}F^{\varepsilon}\right\Vert _{L^{1}(\mathbb{T}^{2})}\lesssim_{\left\Vert \omega_{0}\right\Vert _{L^{\infty}(\mathbb{T}^{2})}}\left\Vert \mathcal{M}F^{\varepsilon}\right\Vert _{L^{1}(\mathbb{T}^{2})},
\end{align*}
where we obtain the first inequality thanks to \Cref{Duality inequality on torus}, and the
second inequality is by \Cref{Endpoint CZ Theorem }.
Now, by \Cref{Orlicz space type estimate -1-1} we get:

\begin{align*}
\left\Vert \mathcal{M}F^{\varepsilon}(t,\cdot)\right\Vert _{L^{1}(\mathbb{T}^{2})}\leq\eta+2\int_{\mathbb{T}^{2}}\left|F^{\varepsilon}(t,x)\right|\log_{+}\left|F^{\varepsilon}(t,x)\right|dx+2\log\left(\frac{1}{\eta}\right)\left\Vert F^{\varepsilon}(t,\cdot)\right\Vert _{L^{1}(\mathbb{T}^{2})}.\label{eq:-9}    
\end{align*}

 In order to bound $\int_{\mathbb{T}^{2}}\left|F^{\varepsilon}(t,x)\right|\log_{+}\left|F^{\varepsilon}(t,x)\right|dx$
by means of $\left\Vert F^{\varepsilon}(t,\cdot)\right\Vert _{L^1(\mathbb{T}^{2})}$,
we need an estimate on $\left\Vert F^{\varepsilon}(t,\cdot)\right\Vert _{L^\infty(\mathbb{T}^{2})}$.
By \Cref{Estimate in terms of initial data } we have

\begin{align*}
 \left|F^{\varepsilon}(t,x)\right|\leq\int_{\mathbb{R}^{2}}\left|\xi\right|^{2}f^{\varepsilon}(t,x,\xi)d\xi+\left\Vert \omega_{0}\right\Vert _{ L^{\infty}(\mathbb{T}^{2})}\int_{\mathbb{R}^{2}}f^{\varepsilon}(t,x,\xi)d\xi   
\end{align*}
\begin{equation}
\hspace{-1.5 cm}\lesssim_{\left\Vert \omega_{0}\right\Vert _{L^{\infty}(\mathbb{T}^{2})}} \left\Vert m_{2}^{\varepsilon}(t,\cdot) \right\Vert_{\infty}+\left\Vert \rho_{f}^{\varepsilon}(t,\cdot)\right\Vert _{L^\infty(\mathbb{T}^{2})}.\label{eq:-10}
\end{equation}
Hence, \Cref{propagation of the density } entails 
\begin{align*}
\left\Vert \mathcal{M}F^{\varepsilon}(t,\cdot)\right\Vert _{L^{1}(\mathbb{T}^{2})}\lesssim&_{\left\Vert \omega_{0}\right\Vert _{L^{\infty}(\mathbb{T}^{2})}}\eta+\log\left(\left\Vert m_{2}^{\varepsilon}(t,\cdot) \right\Vert_{\infty}+\left\Vert \rho_{f}^{\varepsilon}(t,\cdot)\right\Vert _{L^\infty(\mathbb{T}^{2})}\right)\left\Vert F^{\varepsilon}(t,\cdot)\right\Vert _{L^1(\mathbb{T}^{2})}\\&+\log\left(\frac{1}{\eta}\right)\left\Vert F^{\varepsilon}(t,\cdot)\right\Vert _{L^1(\mathbb{T}^{2})}\\
\lesssim&_{\left\Vert \omega_{0}\right\Vert _{L^{\infty}(\mathbb{T}^{2})}}\eta+\log\left(\Gamma_{\varepsilon}^{2}(t)+\Gamma_{\varepsilon}(t)\right)E_{\varepsilon}(t)+\log\left(\frac{1}{\eta}\right)E_{\varepsilon}(t).
\end{align*}
To conclude

\begin{equation}
|I_{1}(t)|\lesssim_{\left\Vert \omega_{0}\right\Vert _{L^{\infty}(\mathbb{T}^{2})}}\left(\eta+\log\left(\Gamma^{2}_{\varepsilon}(t)+\Gamma_{\varepsilon}(t)\right)E_{\varepsilon}(t)+\log\left(\frac{1}{\eta}\right)E_{\varepsilon}(t)\right).\label{eq:-1}
\end{equation}

\textbf{Step 2}.\textbf{ Estimate on $I_{2}$.} The estimate for $I_{2}$
is obtained via considerations similar to the ones in
Step 1. We have 

\[
I_{2}(t)= \frac{1}{\varepsilon}\int_{\mathbb{T}^{2}}\mathrm{div}(u)\left|\nabla V\star(\rho_{f}^{\varepsilon}-1)(t,x)\right|^{2}dx.
\]
Thanks to \Cref{Endpoint CZ Theorem } and \Cref{Duality inequality on torus} we have:

\begin{align*}
\left|\int_{\mathbb{T}^{2}}\mathrm{div}(u)\left|\nabla V\star(\rho_{f}^{\varepsilon}-1)(t,x)\right|^{2}dx\right|
& \leq\left\Vert \mathrm{div}(u(t,\cdot))\right\Vert _{\mathrm{BMO(\mathbb{T}^{2})}}\left\Vert  \mathcal{M}\left|\nabla V\star(\rho_{f}^{\varepsilon}-1)\right|^{2}(t,\cdot)\right\Vert _{{L}^{1}(\mathbb{T}^{2})}\\&\lesssim_{\left\Vert \omega_{0}\right\Vert _{L^{\infty}(\mathbb{T}^{2})}}  \left\Vert \mathcal{M}\left|\nabla V\star(\rho_{f}^{\varepsilon}-1)\right|^{2}(t,\cdot)\right\Vert _{L^{1}(\mathbb{T}^{2})}.  
\end{align*}
Notice that, using almost the same analysis as in the proof of \Cref{ESTIMATE ON FORCE FIELD }, we obtain the bound 

\[
\left\Vert \left|\nabla V\star(\rho_{f}^{\varepsilon}-1)(t,x)\right|^{2}\right\Vert _{L^{\infty}(\mathbb{T}^{2})}\lesssim1+\left\Vert \rho_{f}^{\varepsilon}(t,\cdot)\right\Vert _{L^1(\mathbb{T}^{2})}^{2}+\left\Vert \rho_{f}^{\varepsilon}(t,\cdot)\right\Vert _{L^\infty(\mathbb{T}^{2})}^{2}\lesssim1+\Gamma_{\varepsilon}^{2}(t).
\]
Thanks to \Cref{Orlicz space type estimate -1-1} and proceeding as in Step 1 we find
\begin{equation}
|I_{2}(t)|\lesssim_{\left\Vert \omega_{0}\right\Vert _{L^{\infty}(\mathbb{T}^{2})}}\eta+\log\left(1+\Gamma_{\varepsilon}^{2}(t)\right)E_{\varepsilon}(t)+\log\left(\frac{1}{\eta}\right)E_{\varepsilon}(t).\label{eq:-7-1}
\end{equation}

\textbf{Step 3}. \textbf{Estimate on $I_{3}$. }Let $\varphi(t,x)= A(t,x)\cdot u(t,x)$. We have 

\begin{align*}
\left|\int_{\mathbb{T}^{2}}\varphi(t,x)\left(\rho^{\varepsilon}(t,x)-1\right)dx\right|&\leq\underset{t\in[0,T]}{\sup}\left\Vert \nabla\varphi(t,\cdot)\right\Vert _{2}\left\Vert \rho^{\varepsilon}(t,\cdot)-1\right\Vert _{H^{-1}(\mathbb{T}^{d})}\\
&\leq\varepsilon\underset{t\in[0,T]}
{\sup}\left\Vert \nabla\varphi(t,\cdot)\right\Vert _{2}^2 + \varepsilon \norme{\nabla \Phi^\varepsilon(t)}_{L^{2}(\mathbb{T}^{2})}^2\\
&\leq\varepsilon\underset{t\in[0,T]}{\sup}\left\Vert \nabla\varphi(t,\cdot)\right\Vert _{2}^2+  E_{\varepsilon}(t)\lesssim_{\left\Vert \omega_{0}\right\Vert _{L^{\infty}(\mathbb{T}^{2})}}\varepsilon+E_{\varepsilon}(t).    
\end{align*}


\textbf{Step 4}. \textbf{Conclusion}. Altogether, we obtain 
\[
\dot{E_{\varepsilon}}(t)\lesssim_{\left\Vert \omega_{0}\right\Vert _{L^{\infty}(\mathbb{T}^{2})}}\varepsilon+\eta+\left(\log\left(\Gamma^{2}_{\varepsilon}(t)+\Gamma_{\varepsilon}(t)\right)+\log\left(1+\Gamma_{\varepsilon}^{2}(t)\right)+\log\left(\frac{1}{\eta}\right)\right)E_{\varepsilon}(t).
\]
Putting $\overline{\Gamma}_{\varepsilon}\coloneqq \underset{t\in [0,T]}{\mathrm{sup}}{\Gamma}_{\varepsilon}(t)$, we get by  Gr\"onwall's inequality   

\[
E_{\varepsilon}(t)\leq\frac{\left(\overline{\Gamma}_{\varepsilon}^{2}+\overline{\Gamma}_{\varepsilon}\right)^{CT}\left(1+\overline{\Gamma}_{\varepsilon}^{2}\right)^{CT}}{\eta^{CT}}\left(\varepsilon+\eta+E_{\varepsilon}(0)\right)\leq\frac{1}{\varepsilon^{\gamma CT}\eta^{CT}}\left(\varepsilon+\eta+E_{\varepsilon}(0)\right),
\]
where $C=C(\left\Vert \omega_{0}\right\Vert _{L^{\infty}(\mathbb{T}^{2})})$ and $\gamma=\gamma(\alpha)>0$. Choosing $\eta\sim\varepsilon^{\delta}$ for some arbitrary $\delta>0$ and $T_{\ast}=T_{\ast}(\left\Vert \omega_{0}\right\Vert _{L^{\infty}(\mathbb{T}^{2})},\beta,\alpha,\delta)>0$
sufficiently small, and using \textbf{H4}, we finally get $\underset{t\in[0,T_{\ast}]}{\sup}E_{\varepsilon}(t)\underset{\varepsilon\rightarrow0}{\rightarrow}0$. 
\begin{flushright}
$\square$
\par\end{flushright}

\begin{rem}
The convergence $\underset{t\in[0,T]}{\sup}E_{\varepsilon}(t)\underset{\varepsilon\rightarrow0}{\rightarrow}0$ implies $\rho^{\varepsilon}(t,\cdot)\underset{\varepsilon\rightarrow0}{\rightharpoonup}1$ and $J^{\varepsilon}(t,\cdot)\underset{\varepsilon\rightarrow0}{\rightharpoonup}u(t,\cdot)$ weakly in the sense of measures on $[0,T]$.     
\end{rem}
\section{\label{Construction-of-initial data } Construction of well prepared
initial data}
Take
\[
f_{0}^{\varepsilon}(x,\xi)\coloneqq\frac{1}{\varepsilon^{\beta}}\exp\left(-\frac{\left|\xi-u_{0}(x)\right|^{2}}{2\varepsilon^{\beta}}\right).
\]
By variable change we see that 
\[
\int_{\mathbb{R}^{2}}f_{0}^{\varepsilon}(x,\xi)d\xi=1
\]
hence 

\[
\iint_{\mathbb{T}^{2}\times\mathbb{R}^{2}} f_{0}^{\varepsilon}(x,\xi)d\xi dx=1,
\]
so that \textbf{H1} holds. This also shows that for this choice the
interaction part vanishes identicaly 
\[
E_{\varepsilon}^{2}(0)=0.
\]
As for the kinetic part, we compute 

\[
E_{\varepsilon}^{1}(0)=\frac{1}{2}\iint_{\mathbb{T}^{2}\times\mathbb{R}^{2}}\left|\xi-u_{0}(x)\right|^{2}f_{0}^{\varepsilon}(x,\xi)dxd\xi=\frac{1}{\varepsilon^{\beta}}\iint_{\mathbb{T}^{2}\times\mathbb{R}^{2}}\left|\xi\right|^{2}\exp\left(-\frac{\left|\xi\right|^{2}}{2\varepsilon^{\beta}}\right)dxd\xi
\]

\begin{equation}
=\varepsilon^{\beta}\iint_{\mathbb{T}^{2}\times\mathbb{R}^{2}}\left|\xi\right|^{2}\exp\left(-\frac{\left|\xi\right|^{2}}{2}\right)d\xi dx=O(\varepsilon^{\beta}).\label{kinetic part-1-1}
\end{equation}
This shows that $E_{\varepsilon}(0)=O(\varepsilon^{\beta})$, which
is requirement \textbf{H2}. To verify \textbf{H3}, note that 

\[
\left\Vert f_{0}^{\varepsilon}\right\Vert _{L^{\infty}\cap L^{1}}\leq\frac{1}{\varepsilon^{\beta}}+1
\]
and 

\begin{equation}
\left\Vert \left|\xi\right|^{k_{0}}f_{0}^{\varepsilon}\right\Vert _{L^{\infty}\cap L^{1}}\leq\frac{1}{\varepsilon^{\beta}}+\frac{1}{\varepsilon^{\beta}}\iint_{\mathbb{T}^{2}\times\mathbb{R}^{2}}\left|\xi\right|^{k_{0}}\exp\left(-\frac{\left|\xi-u_{0}(x)\right|^{2}}{2\varepsilon^{\beta}}\right)d\xi dx.\label{eq:-5}
\end{equation}
After changing variables the integral on the right hand side of (\ref{eq:-5})
writes 
\begin{align*}
\frac{1}{\varepsilon^{\beta}}\iint_{\mathbb{T}^{2}\times\mathbb{R}^{2}}\left|\xi+u_{0}\right|^{k_{0}}\exp\left(-\frac{\left|\xi\right|^{2}}{2\varepsilon^{\beta}}\right)d\xi dx\lesssim_{k_{0}}&\frac{1}{\varepsilon^{\beta}}\iint_{\mathbb{T}^{2}\times\mathbb{R}^{2}}\left|\xi\right|^{k_{0}}\exp\left(-\frac{\left|\xi\right|^{2}}{2\varepsilon^{\beta}}\right)d\xi dx\\
&+\left\Vert u_{0}\right\Vert _{\infty}^{k}\frac{1}{\varepsilon^{\beta}}\iint_{\mathbb{T}^{2}\times\mathbb{R}^{2}}\exp\left(-\frac{\left|\xi\right|^{2}}{2\varepsilon^{\beta}}\right)d\xi dx\\&=\left\Vert u_{0}\right\Vert _{\infty}^{k_{0}}+\varepsilon^{\frac{\beta k_{0}}{2}}\iint_{\mathbb{T}^{2}\times\mathbb{R}^{2}}\left|\xi\right|^{k_{0}}\exp\left(-\frac{\left|\xi\right|^{2}}{2}\right)d\xi dx\\&\lesssim_{k_{0},\left\Vert u_{0}\right\Vert _{\infty}}1.
\end{align*}
So to conclude this shows that 
\[
\left\Vert (1+\left|\xi\right|^{k_{0}})f_{0}^{\varepsilon}\right\Vert _{L^{\infty}\cap L^{1}}\lesssim_{k_{0},\left\Vert u_{0}\right\Vert _{\infty}}\varepsilon^{-\beta}.
\]

\appendix

\section{Well-posedness of the Vlasov--Poisson system on the two-dimensional torus}\label{sec:appen}
For completeness, we recall the well-posedness theory of the Vlasov--Poisson system on $\mathbb{T}^{2}\times\mathbb{R}^{2}$. Here the Debye length doesn't play any role, thus we set $\varepsilon=1$. 
We follow the presentation of the well-posedness theory for \eqref{eq:Vlasov Equation} in \cite{GI21}.
\begin{thm}
\textup{\label{thm:(Existence-for-Vlasov--Poisson)} (Well-posedness for
Vlasov--Poisson on the $2$ dimensional torus)} Let $f_{0}\in L^{\infty}\cap L^{1}(\mathbb{T}^{2}\times\mathbb{R}^{2})$
be a probability density. Assume that 
\begin{enumerate}
    \item There are some $k_{0}>2$ and $\kappa_{0}>0$ such
that for a.e. $(x,\xi)\in\mathbb{T}^{2}\times\mathbb{R}^{2}$ it holds
that 

\[
f_{0}(x,\xi)\leq\frac{\kappa_{0}}{1+\left|\xi\right|^{k_{0}}}.
\]
    \item  There are some $m_{0}>2$ and $\mu_{0}>0$ such that
\[
\iint_{\mathbb{T}^{2}\times\mathbb{R}^{2}}\left|\xi\right|^{m_{0}}f_{0}(x,\xi)dxd\xi\leq\mu_{0}.
\]
\end{enumerate}

Then, there exists a unique solution $f\in C([0,\infty);\mathcal{P}(\mathbb{T}^{2}\times\mathbb{R}^{2}))$
to \eqref{eq:Vlasov Equation}.
\end{thm}

The assumptions on $f_0$ in the theorem above correspond to \textbf{H1-H2} in \Cref{main thm}. These assumptions are sufficient to define a unique weak solution to the Vlasov--Poisson system \cite{Batt Rein, Lions Perthame}, which also satisfies propagation of velocity moments. This is an important property which also implies propagation of regularity. Since we couldn't find a detailed proof of propagation of velocity moments for solutions to \eqref{eq:Vlasov Equation} on the two-dimensional torus, we write it below. 

\subsubsection*{Propagation of velocity moments}
We consider the velocity moment of a weak solution $f$ to \eqref{eq:Vlasov Equation} of order $k\geq 0$ defined by
\begin{equation}
    M_k(t)\coloneqq \iint_{\mathbb{T}^2 \times \mathbb{R}^2} \abs{\xi}^k f(t,x,v)dxd\xi.
\end{equation}
We also recall that we defined the partial moment for $k\geq 0$ by
\begin{equation*}
m_k(f)(x):=\int_{\mathbb{R}^2} \abs{\xi}^k f d\xi.
\end{equation*}
We will show that $M_k(t)$ is bounded for all time by showing it satisfies a Gr\"onwall inequality.
Before showing this, we recall a classical inequality for velocity moments, which is a generalization of \Cref{interpolation inequality }.
\begin{lem}\label{lem:ineq_mk}
	Let $0\leq k'\leq k < \infty$ and $r=\frac{k+2}{k'+2}$. 
	If $f\in L^\infty(\mathbb{T}^2 \times \mathbb{R}^2)$ with $M_k(f) < \infty$ then $m_{k'}(f)\in L^r(\mathbb{T}^2)$ and 
	
	\begin{equation}
	\norme{m_{k'}(f)}_r \leq c \norme{f}_\infty^{\frac{k-k'}{k+2}}M_k(f)^{\frac{k'+2}{k+2}}
	\end{equation} where $c=c(k,k',p)>0$.
\end{lem}
First we notice that $M_k$ verifies the following differential inequality
\begin{align*}
\abs{\frac{d}{dt} M_k(t)} & =	\abs{\iint \abs{\xi}^k (-v\cdot\nabla_{x}f-\nabla_x \Phi\cdot \nabla_{v} f) d\xi dx}\\
& = \abs{\iint \abs{\xi}^k \mathop{\rm div_v}\left(\nabla_x \Phi f\right)d\xi dx} = \abs{\iint k\abs{\xi}^{k-2} v\cdot \nabla_x \Phi f d\xi dx}\\
& \leq \iint k\abs{\xi}^{k-1} f d\xi \abs{\nabla_x \Phi} dx \leq k \norme{\nabla_x \Phi(t)}_{k+2}\norme{\int_{\mathbb{R}^2} f(t,\cdot,\xi) d\xi}_{\frac{k+2}{k+1}}.
\end{align*}

By \Cref{lem:ineq_mk} with $k'=k-1$, we have
\begin{equation*}
    \norme{\int_{\mathbb{T}^2} \abs{\xi}^{k-1}f(t,\cdot,\xi) d\xi}_{\frac{k+2}{k+1}} \leq C M_k(t)^\frac{k+1}{k+2}.
\end{equation*}
Therefore we have 
\begin{equation*}
    \dot{M}_k(t) \leq k C \norme{\nabla_x \Phi(t)}_{k+2} M_k(t)^\frac{k+1}{k+2}.
\end{equation*}

To close the Gr\"onwall inequality, we aim to obtain $\norme{\nabla_x \Phi(t)}_{k+2} \leq M_k(t)^\frac{1}{k+2}$. Hence we need to study the electric field $\nabla_x \Phi$ more carefully. We recall that $\nabla_x \Phi= -\nabla_x G \ast \rho$ where $G$ is the Green function for the negative Laplacian on the torus. Indeed $G$ is the solution to $-\Delta G=\delta_0-1$ and is given by
\begin{equation*}
    G(x)=-\frac{1}{2\pi} \log\abs{x}+G_0(x)
\end{equation*}
where $G_0 \in C^\infty(\overline{B_{1/4}(0)})$.

Therefore $G$ has the same singularity as the Green function on the full space. Thus, $\nabla G $ belongs to the Lorentz space $\in L^{2,\infty}(\mathbb{T}^2)$. Hence, using the weak Young inequality we obtain
\begin{equation*}
    \norme{\nabla_x \Phi(t)}_{k+2}\leq C \norme{\rho(t)}_\frac{2k+4}{k+4}.
\end{equation*}

Now applying \Cref{lem:ineq_mk} with $k'=0$ we have
\begin{equation*}
    \norme{\rho(t)}_\frac{k+2}{2} \leq C M_k(t)^\frac{2}{k+2}.
\end{equation*}
Since $\frac{2k+4}{k+4} \leq \frac{k+2}{2}$ we have the following H\"older's inequality
\begin{equation*}
    \norme{\rho(t)}_\frac{2k+4}{k+4} \leq \norme{\rho(t)}_1^{1-\theta} \norme{\rho(t)}_{\frac{k+2}{2}}^{\theta},
\end{equation*}
and after carefully computing the exponents in the inequality we finally obtain $1-\theta=\theta=\frac{1}{2}$.
Therefore
\begin{equation*}
    \norme{\nabla_x \Phi(t)}_{k+2} \leq C \norme{\rho(t)}_\frac{2k+4}{k+4} \leq C \norme{\rho(t)}_\frac{k+2}{2}^\frac{1}{2} \leq C M_k(t)^\frac{1}{k+2}.
\end{equation*}
This inequality allows us to close the Gr\"onwall estimate and so $M_k(t)$ is bounded for all time. We notice that, once we obtain the propagation of velocity moments, we have a bound on the density that implies that the electric field is bounded. Once the electric field is bounded, Assumption 1 is verified for all times.
For more details see \cite{Lions Perthame, GI21, GI21bis, Griffin-Pickering Iacobelli 2023}.
\subsubsection*{Uniqueness}
Thanks to Assumption 1 in \Cref{thm:(Existence-for-Vlasov--Poisson)}, and to the propagation of moments, we can show that $\rho \in L^\infty_{loc}(\mathbb{R}^+;L^\infty(\mathbb{T}^2))$ for all times. Therefore, the uniqueness criterion established in Loeper's work \cite{L06} is verified, which concludes the proof of \Cref{thm:(Existence-for-Vlasov--Poisson)}.

\subsubsection*{Acknowledgments}

IBP was supported by the EPSRC grant number EP/V051121/1.
This work was also supported by the Advanced Grant Nonlocal-CPD (Nonlocal PDEs for Complex Particle Dynamics: Phase Transitions, Patterns and Synchronization) of the European Research Council Executive Agency (ERC) under the European Union's Horizon 2020 research and innovation programme (grant agreement No. 883363). 

AR was supported by the NCCR SwissMAP which was funded by the Swiss National Science Foundation grant number 205607. AR would like to thank the Swiss National Science Foundation for its financial support.

We would like to thank Pierre-Emmanuel Jabin for many stimulating discussions which have greatly contributed to this work, and to Matthew Rosenzweig for his interest in this work.  
We thank the anonymous referee for comments which improved the quality of this work.

\end{document}